\documentclass{cedram-alco-modified}
\usepackage{cleveref}
\usepackage{xcolor}
\usepackage{pgfplots}
\pgfplotsset{compat=1.17}
\usepackage{mathrsfs}
\usetikzlibrary{arrows}
% ...
\newcommand{\ordm}{\operatorname{ord-match}}
\newcommand{\indm}{\operatorname{ind-match}}
\newcommand{\reg}{\operatorname{reg}}
\newcommand{\minm}{\operatorname{min-match}}
\newcommand{\m}{\operatorname{match}}
\newcommand{\depth}{\operatorname{depth}}

\def\qedbox{\hbox{\vbox{\hrule\hbox{\vrule\kern3pt\vbox{\kern6pt}\kern3pt\vrule}\hrule}}}%ineqalignno, use: "& \qedbox \cr"
\def\qed{\unskip\hfill\qedbox\vskip3mm} % independent
\def\eqed{\eqno\hbox{\quad\qedbox}} % in other equations
\def\drlabel#1{\label{#1}}
\title
[Induced, ordered matchings, and regularity]
{Induced matching, ordered matching and Castelnuovo-Mumford regularity of bipartite graphs}

\author[\initial{A.} \middlename{V.} Jayanthan]{\firstname{A.} \middlename{V.}
\lastname{Jayanthan}}
\address{Department of Mathematics, I.I.T. Madras, Chennai, INDIA - 600036.}
\email{jayanav@iitm.ac.in}
\author[\initial{S.} \middlename{Amin} Seyed Fakhari]
{\firstname{Seyed} \middlename{Amin} \lastname{Seyed Fakhari}}
\address{Departamento de Matem\'aticas, Universidad de los Andes, Bogot\'a, Colombia.}
\email{s.seyedfakhari@uniandes.edu.co}
\author[\initial{I.} Swanson]{\firstname{Irena} \lastname{Swanson}}
\address{Department of Mathematics, Purdue University, West Lafayette, IN 47907, USA.}
\email{irena@purdue.edu}
\author[\initial{S.} Yassemi]{\firstname{Siamak} \lastname{Yassemi}}
\address{Department of Mathematical Sciences, Carnegie Mellon University, 5000 Forbes Avenue, Pittsburg, PA 15213}
\email{syassemi@andrew.cmu.edu}

\thanks{
Part of the work was done when the first author was visiting Purdue University. He would like to thank the Department for funding the visit and for a great hospitality. The research of the second author is supported by a FAPA grant from the Universidad de los Andes.
}

\keywords{Induced matching, ordered matching, Castelnuovo-Mumford regularity, depth, edge ideal, cover ideal}

\subjclass{10X99, 14A12, 11L05}

\begin{document}

%% Abstracts must be placed before \maketitle
\begin{abstract}
Let $G$ be a finite simple graph and let $\indm(G)$ and $\ordm(G)$ denote the induced matching number and the ordered matching number of $G$, respectively.
We characterize all bipartite graphs $G$ with $\indm(G)=\ordm(G)$. We establish the Castelnuovo-Mumford regularity of powers of edge ideals and depth of powers of cover ideals for such graphs. We also give formulas for the count of connected non-isomorphic spanning subgraphs of $K_{m,n}$ for which $\indm(G) = \ordm(G) = 2$, with an explicit expression for the count when $m \in \{2,3,4\}$ and $m \le n$.
\end{abstract}

\newtheorem{example}[cdrthm]{Example}
\newtheorem{examples}[cdrthm]{Examples} %Does NOT seem to work for Cref!!!
\newtheorem{proposition}[cdrthm]{Proposition}
\newtheorem{definition}[cdrthm]{Definition}
\newtheorem{corollary}[cdrthm]{Corollary}
\newtheorem{remarks}[cdrthm]{Remarks}
\newtheorem{notation}[cdrthm]{Notation}

\maketitle

\section{Introduction}\label{sec1}

Let $G$ be a finite simple graph on the vertex set $V(G)= \{x_1,\ldots,x_d\}$ and edge set $E(G)$. We identify the vertices to variables and consider the polynomial ring $S=K[x_1,\ldots,x_d]$, where $K$ is a field. The \textit{edge ideal} of $G$ is defined as $I(G) = \langle \{x_ix_j ~:~ x_ix_j \in E(G)\}\rangle \subset S$. Ever since the introduction of the edge ideal by Villarreal in \cite{V90}, researchers have been trying to understand the interplay between the combinatorial properties of graphs and the algebraic properties of the associated edge ideals. One particular invariant, the Castelnuovo-Mumford regularity, has received much of the attention, compared to other invariants and properties. Several upper and lower bounds for the regularity of edge ideals were obtained by several researchers, see \cite{BBH19} and references therein. Whenever there is an upper and a lower bound for an invariant, it is natural to ask what are some necessary conditions and sufficient conditions for these two bounds to coincide, and structurally understand those objects for which these two bounds are equal. In this article, we address this question for the upper bound of ordered matching number and the lower bound of induced matching number.

%Let $G$ be a finite simple graph. 
Computing or bounding the Castelnuovo-Mumford regularity of the associated edge ideal and its powers, in terms of combinatorial data associated with $G$, has been a very active area of research for the past couple of decades. Bounds using several \textit{matching numbers} have been obtained for the regularity. For graph $G$, let $\indm(G), \ordm(G), \minm(G)$ and $\m(G)$ denote induced matching number, ordered matching number, minimum matching number and matching number, respectively (see Section 2 for the definitions). 

%If ind-match$(G)$ and match$(G)$ denote the induced matching number and the matching number of $G$ respectively, then 
It is known that \[\indm(G) \leq \reg(S/I(G))\leq \{\ordm(G), \minm(G)\} \leq \m(G),\]
%\[\indm(G) \leq \reg(R/I(G))\leq \ordm(G) \leq \text{match}(G).\]
where the first inequality was proved by Katzman, \cite{K06}, the second inequality can be found in~\cite{W14} (for $\minm(G)$) and \cite{CV11} (for $\ordm(G))$ and the third inequality follows from the definition.
A graph $G$ is said to be a {\it Cameron-Walker graph} if $\indm(G) = \m(G)$. This is a class of graphs which is well studied from both combinatorial and algebraic perspectives, \cite{CW05,HHKO15,HKMT21}. 
In \cite{HHKT16}, Hibi et al. studied graphs with $\indm(G) = \minm(G)$. They gave a structural characterization of graphs satisfying $\indm(G) = \minm(G)$. 
In this article, we study graphs satisfying $\indm(G) = \ordm(G)$. %It may be noted that an ordered matching depends on the ordering of the edges. Due to this, the invariant is more difficult to control, compared to the other matching invariants, while adding or deleting vertices and edges. %Unlike in the case where the graph has dominated induced matching (for example, $\indm(G) = \minm(G)$), characterizing graphs $G$ with $\indm(G) = \ordm(G)$ is almost impossible due to the dependence of the ordered matching on the ordering the edges involved. We give two distinct characterization of bipartite graphs with $\ordm(G) \leq 2$, the first one in terms of complements and then in terms of structure of edges. To complete the understanding of all bipartite graphs with $\indm(G) = 2 = \ordm(G)$, we use these characterizations to  the count them. We find that this counting is equivalent to counting some integer sequences and understanding certain summations.

Besides the combinatorial reasons for understanding graphs $G$ with $\indm(G) = \ordm(G)$, there is also an algebraic motivation to understand graphs with this property. It was proved by Cutkosky, Herzog and Trung, \cite{CHT}, and independently by Kodiyalam, \cite{vijay}, that for a homogeneous ideal $I$ in a polynomial ring, $\reg(I^s)$ is a linear polynomial for $s\gg 0$. In the case of edge ideals, there have been extensive research in understanding this function and the polynomial in terms of combinatorial data associated with $G$, (see for example \cite{BBH19} and the references within). 
It was shown in \cite[Theorem 4.5]{selviha} and \cite[Theorem 3.9]{SY23} that for every integer $s\geq 1$,$$2s+\indm(G)-2\leq \reg(S/I(G)^s)\leq 2s+\ordm(G)-2.$$
If $\indm(G)=\ordm(G)$, then this gives an explicit expression for the regularity of powers of the edge ideal.

Classifying all graphs $G$ with $\indm(G) = \ordm(G)$ would be  an extremely hard problem in general, and in this paper we concentrate on classifying bipartite graphs satisfying this property. Another important reason for restricting our attention to the bipartite case is the behavior of the depth function of the cover ideal, see the end of \Cref{sec:prelim}. 

For smaller values of induced and ordered matching, it is easier to handle the corresponding bipartite graphs. First we give graph theoretic characterization for graphs $G$ with  $\indm(G)=\ordm(G)=1$ (Theorem \ref{thmindred1}). We then move on to understand the structure of graphs in terms of the connectivity between the bipartitions. This gives us a classification of all bipartite graphs $G$ with $\indm(G)=\ordm(G)$, (\Cref{thmordindclass}). 
To illustrate that this class of graphs is very different from the class of graphs $G$ with $\indm(G) = \minm(G)$, we construct a class of graphs, $G_{r,m}$, $2\leq r \leq m$, with $\reg(S/I(G_{r,m})) = \indm(G) = \ordm(G)  = r$ and $\minm(G) = m$.

%Since our characterization of graphs  having equal induced and ordered matching numbers is in terms of existence of possible edges between the two partitions of the vertex set, it is natural to ask what are the possible occurrences of these edges. This can be understood by counting the number of such graphs on a given vertex set (once the vertex set is fixed, a graph is characterized by the presence of edges between the vertices). To this end, we start counting the number of graphs $G$ satisfying $\indm(G)=\ordm(G)=2$. This leads us to some interesting connections between the presence of edges and certain number of integer sequences. We obtain a precise expression for the number of bipartite graphs with $\indm(G)=\ordm(G)=2$ in certain cases, \Cref{thmall}.

Our characterization of bipartite graphs with equal induced and ordered matching numbers allows us to count the number of graphs $G$ satisfying $\indm(G)=\ordm(G)\le 2$, see \Cref{thmindred1} and \Cref{thmall}. The work for $\indm(G)=\ordm(G)=2$ leads us to some interesting connections between the presence of edges and certain number of integer sequences. 

The paper is organized as follows. We collect some graph theory essentials in \Cref{sec:prelim}. Characterizing the equality of induced and ordered matching numbers is done in \Cref{sec:indordeq}. In \Cref{sec3}, we introduce some preliminaries to count the bipartite graphs with $\indm(G)=2 = \ordm(G)$, namely we set up the notation for certain integer sequences and their counting. We count all graphs with $\indm(G)=2=\ordm(G)$ in \Cref{sec:count}. We provide the closed forms for certain elementary summations to count these graphs in \Cref{sec:appendix} as an appendix.

%\begin{question}
%Given $0 < q  \leq s, \ell, \leq t$, does there exists a graph $G$ with $\indm(G) = q$, $\ordm(G) = s, {\rm min-match}(G)=\ell$ and {\em match}$(G) = t$?
%\end{question}

%\begin{question}
%Characterize all graphs with $\indm(G) = \ordm(G)$.
%\end{question}
%Amin has found a characterization (will be added here  soon) if $G$ is sequentially Cohen-Macaulay bipartite. A general classification seems difficult. This class contains, $K_n$ for all $n \geq 2$,  $C_3, C_4, C_6$ and no other cycles. One possibility is to look for good graph classes and characterize the equality within the graph class. A suggestion is to restrict ourselves to bipartite graphs.

%\begin{question}
%In arXiv:1301.6779, 
%H\`a-Woodroofe provide a generalization of ind-match and min-match to hypergraphs. Generalize the notion of $\ordm(G)$  to the case of hypergraphs.
%\begin{enumerate}
%    \item If $H$ is a hypergraph, then is $\reg(R/I(H)) \leq \ordm(H)?$
%    \item If $J(H)$ denote the cover ideal of the hypergraph  $H$, then is $\lim \text{depth}R/J(G)^{(k)} = n-\ordm(G)-1$?
%\end{enumerate}
%\end{question}

%\begin{question}
%
%In arxiv:2018.06750v3 Hein-Trung proved that: 

%Let G be a graph. Then,
%$\reg(I(G)^{(s)}) \le 2s+\ordm(G)-1$, for all $s> 1$.

%\vspace{.2in}

%Do we have the following inequality:

%$\reg(I(G)^{(s)}) \le 2s+\minm(G)-1$, for all $s> 1$?
%\end{question}

%\begin{question}
%Characterize all bipartite graphs $G$ satisfying $\indm(G) = \ordm(G)$.
%\end{question}
\section{Preliminaries}\label{sec:prelim}
In this section, we provide the definitions and basic facts which will be used in the next sections.

Let $S=K[x_1, \dots, x_n]$ be the polynomial ring over a field $K$ and let $M$ be a finitely generated graded $S$-module. Suppose that the minimal graded free resolution of $M$ is given by
$$0\rightarrow \cdots \rightarrow \bigoplus_j S(-j)^{\beta _{1,j}(M)}\rightarrow \bigoplus_j S(-j)^{\beta _{0,j}(M)}\rightarrow M\rightarrow 0.$$
The Castelnuovo-Mumford regularity (or simply, regularity) of $M$, denoted by ${\rm reg}(M)$, is defined as
$${\rm reg}(M)={\rm max}\{j-i\mid \beta _{i,j}(M)\neq 0\}.$$Also, the {\it projective dimension} of $M$ is defined to be$${\rm pd}(M)={\rm max}\{i\mid \beta _{i,j}(M)\neq 0 \ {\rm for \ some \ } j\}.$$

For a vertex $x_i\in V(G)$, the {\it neighbor set} of $x_i$ is defined to be the set $N_G(x_i)=\big\{x_j\mid x_ix_j\in E(G)\big\}$. Moreover, the {\it closed neighborhood} of $x_i$ is $N_G[x_i]=N_G(x_i)\cup \{x_i\}$. The cardinality of $N_G(x_i)$ is the {\it degree} of $x_i$ and is denoted by ${\rm deg}_G(x_i)$. A vertex of degree one is called a {\it leaf} of $G$. The graph $G$ is a {\it forest} if it does not have any cycle. The {\it distance} between $x_i$ and $x_j$ in $G$ is defined to be the length of the shortest path between $x_i$ and $x_j$ in $G$. For a subset $W \subset V(G)$, $G \setminus W$ denotes the induced subgraph of $G$ on the vertex set $V(G) \setminus W$. A subset $A$ of $V(G)$ is said to be an {\it independent subset} of $G$ if there are no edges among the vertices of $A$. The graph $G$ is called {\it unmixed} (or {\it well-covered}) if all maximal independent sets of $G$ have the same cardinality.

A \textit{matching} in a graph is a subgraph consisting of pairwise disjoint edges. The cardinality of the largest matching in $G$ is the \textit{matching number} of $G$ and is denoted by $\m(G)$. A matching in $G$ is said to be a \textit{maximal matching} if it is not properly contained in any matching of $G$. The \textit{minimum matching number} of $G$, denoted by $\minm(G)$, is the minimum cardinality of a maximal matching in $G$. A matching is said to be an \textit{induced matching} if none of the edges in the matching are joined by an edge in $G$. The largest size of an induced matching in $G$ is called the \textit{induced matching number} of $G$, denoted by $\indm(G)$. A graph $G$ is called a \textit{Cameron-Walker} graph if $\indm(G) = \m(G)$.

A set $A = \{x_{i1} x_{i2} \in E(G) \mid 1 \leq i \leq r\}$ is said to be an \textit{ordered matching}, \cite{CV11}, if 
\begin{enumerate}
    \item $A$ is a matching in $G$,
    \item $\{x_{i1}  \mid 1 \leq  i \leq r \}$ is an independent set,
    \item if $x_{i1} x_{j2} \in E(G)$, then $i \leq j$.
\end{enumerate}
The \textit{ordered matching number} of $G$, denoted by $\ordm(G)$, is defined to be \[\ordm(G):= \max\{|A| :~ A \text{ is an  ordered matching of } G\}.\]

As already written in the introduction, \[\indm(G) \leq \reg(S/I(G))\leq \{\ordm(G), \minm(G)\} \leq \m(G),\]
%\[\indm(G) \leq \reg(S/I(G))\leq \ordm(G) \leq \text{match}(G).\]
There are several examples with inequality $\minm(G)\leq\ordm(G)$ and other examples with inequality $\ordm(G)\leq\minm(G)$.

%Let $V = \{x_1, \ldots, x_n\}$. A \textit{simplicial complex} on $V$, denoted by $\Delta_V$, is a collection of subsets of $V$ such that $\{x_i\} \in V$ for all $i=1,\ldots,n$ and if $F \in \Delta$ and $G \subset F$, then $G \in \Delta$.

%In a graph $G$, a subset $W$ of $V(G)$ is said to be an \textit{independent} set if none of the vertices in $W$ are adjacent to each other. %The collection of all independent subsets of $V(G)$ is a simplicial complex, called \textit{independence complex}, denoted by $\Delta_G$.
%
%\begin{definition}
%A simplicial complex $\Delta$ is said to be \textit{shellable} if the facets (maximal faces) of $\Delta$ can be ordered $F_1,\cdots, F_n$ such that for all $1\le i<j\le n$, there exists some $v\in F_j\setminus F_i$ and some $\ell\in\{1,\cdots, j-1\}$ with $F_j\setminus F_{\ell}=\{v\}$. 
%A graph $G$ is called shellable if the independence complex $\Delta_G$ is shellable.
A graph $G$ is said to be {\it Cohen-Macaulay} (resp. {\it sequentially Cohen-Macaulay}) if $S/I(G)$ is Cohen-Macaulay (resp. sequentially Cohen-Macaulay).

A {\it bipartite graph} $G$ is a graph with $V(G) = X \sqcup Y$ and $E(G) \subset X \times Y$. If $|X| = m$ and $|Y| = n$ and $E(G) = X \times Y$, then we say that $G$ is a complete bipartite graph, and we denote $G$ by $K_{m,n}$. If $G$ is a bipartite graph, then $G^{bc}$, called the \textit{bipartite complement}, is the bipartite graph with $V(G^{bc}) = V(G)=X \sqcup Y$ and $xy \in E(G^{bc})$ if and only if $xy \notin E(G)$.
A subgraph $H$ of a graph $G$ is said to be a \textit{spanning subgraph} if $V(H) = V(G).$ If $G = K_{m,n}$, then the set of connected spanning subgraphs of $G$ is precisely the set of all connected bipartite graphs on $X \sqcup Y$, where $|X|=m$ and $|Y|=n$.

For the rest of the article, $G$ always denotes a bipartite graph, without isolated vertices, on the finite vertex set $V(G) =  X \sqcup Y$.
%\end{definition}

It was proved by Brodmann \cite{Brod79} that for any homogeneous ideal $I$ in a graded ring~$R$, $\depth(R/I^k)$ is a constant for $k\gg 0$.  One consequence of ind-match and ord-match being equal is the constancy from the start of the depth function of powers of the {\it cover ideal} $J(G) = \bigcap_{x_ix_j \in E(G)} (x_i,x_j)$ of $G$, as we prove next.

\begin{theorem}\drlabel{thmcover}
Assume that $G$ is a bipartite graph with $\indm(G)=\ordm(G)$ and suppose $d=|V(G)|$. Let $J(G)$ denote the cover ideal of $G$. Then for every integer $k\geq 1$, we have$${\rm depth}(S/J(G)^k)=d-\indm(G)-1.$$  
\end{theorem}
\begin{proof}
Since ${\rm reg}(I(G))=\indm(G)+1$ by inequalities in the Introduction, it follows from Terai's theorem \cite[Proposition 8.1.10]{hh2010} that the projective dimension of $S/J(G)$ is equal to $\indm(G)+1$. Thus, Auslander-Buchsbaum formula implies that ${\rm depth}(S/J(G))=d-\indm(G)-1$. On the other hand, it follows from \cite[Corollary 2.6]{GRV05} and \cite[Theorem 3.2]{HKTT17} that$$d-\indm(G)-1={\rm depth}(S/J(G))\geq {\rm depth}(S/J(G)^2)\geq {\rm depth}(S/J(G)^3)\geq \cdots.$$Moreover, we know  from \cite[Theorem 3.4]{HKTT17} (see also \cite[Theorem 3.1]{FAKHARI17}) that ${\rm depth}(S/J(G)^k)=d-\ordm(G)-1$ for any $k\gg 0$. Since $\ordm(G)=\indm(G)$, the assertion follows from the above inequalities. 
\end{proof}

\section{Equality of induced and ordered matching numbers}\label{sec:indordeq}

In this section, we characterize all bipartite graphs $G$ with $\ordm(G)=\indm(G)$. Before proving our general characterization (Theorem \ref{thmordindclass}), we restrict ourselves to a special family of bipartite graphs for which the characterization has simpler formulation comparing with the general case. More precisely, we consider the class of sequentially Cohen-Macaulay bipartite graphs $G$. In \cite{FH, adam, VV08}, the authors studied the sequential Cohen-Macaulayness of $S/I(G)$ in terms of the combinatorial properties of $G$. Here, we study it in terms of the matching numbers. In the following theorem, we show that for a bipartite sequentially Cohen-Macaulay graph $G$, the equality $\ordm(G)=\m(G)$ holds. As a consequence of this equality, we are able to characterize sequentially Cohen-Macaulay bipartite graphs $G$ with $\ordm(G)=\indm(G)$.  %It was shown in \cite{HHKO15} that every Cameron-Walker graph is shellable. Here we characterize the shellability for graphs with $\indm(G) = \ordm(G).$
%We begin our study by showing that the depth function, $\depth(R/J(G)^k)$, is constant for all $k \geq 1$ when $\indm(G) = \ordm(G)$ for a bipartite graph $G$. %For a graph $G$ on the vertex set $V(G)$ and edge set $E(G)$, let $J(G) = \bigcap_{\{x_i,x_j\}\in E(G)} (x_i,x_j) \subset R$ denote the \textit{cover ideal} of $G$.

\begin{theorem}\drlabel{thmseqCM}
Let $G$ be a bipartite graph. If $G$ is sequentially Cohen-Macaulay, then $\ordm(G)={\rm match}(G)$. In particular, for a sequentially Cohen-Macaulay bipartite graph $G$, we have
$\indm(G)=\ordm(G)$ if and only if $G$ is a Cameron-Walker graph.
\end{theorem}

\begin{proof}
We prove the equality $\ordm(G)={\rm match}(G)$ by induction on $|V(G)|$. If $|V(G)|=2$, then $\ordm(G)=\m(G)=1$. Assume by induction that if $H$ is a sequentially Cohen-Macaulay bipartite graph with $|V(H)| < |V(G)|$, then $\ordm(H)=\m(H)$. By \cite[Corollary 3.11]{VV08}, there is a leaf $x\in V(G)$ such that $G\setminus N_G[x]$ is sequentially Cohen-Macaulay. Using \cite[Lemma 2.1]{SF16} and the induction hypothesis we have$$\ordm(G)\geq \ordm(G\setminus N_G[x])+1={\rm match}(G\setminus N_G[x])+1={\rm match}(G)$$ where the last equality follows from the fact that $x$ is a leaf of $G$. Thus, $\ordm(G)={\rm match}(G)$. The second assertion follows by observing that $G$ is Cameron-Walker if and only if $\indm(G)=\m(G)$.
%{\color{red} It may be a good idea to indicate the proof of $\indm(G)=\ordm(G)$ implies $G$ is Cameron-Walker.}
\end{proof}

The converse of the above Theorem does not hold. In fact, the following example shows that for each integer $k\geq 3$, there is a non-sequentially Cohen-Macaulay bipartite graph $G$ with ${\rm match}(G)=\ordm(G)=k$. 

\begin{example}\drlabel{exmatchsell}
For any integer $k\geq 3$, let $G_k$ be the graph obtained from a $4$-cycle graph by attaching a path of length $2k-3$ to exactly one of its vertices. Using induction on $k$, We show that $G_k$ is not sequentially Cohen-Macaulay and $\ordm(G_k)={\rm match}(G_k)=k$. It is easy to see that $\ordm(G_3)={\rm match}(G_3)=3$. Moreover, let $x$ be the unique leaf of $G_3$ and let $y$ be the unique neighbor of $x$. Then $G_3\setminus N_{G_3}[y]$ is the $4$-cycle graph that is not sequentially Cohen-Macaulay. Hence, by \cite[Corollary 3.11]{VV08}, the graph $G_3$ is not sequentially Cohen-Macaulay. Now, suppose that $k\geq 4$. Let $z$ be the unique leaf of $G_k$. Then $G_k\setminus N_{G_k}[z]$ is isomorphic to $G_{k-1}$ that is not sequentially Cohen-Macaulay by induction hypothesis. Moreover, since $z$ is a leaf of $G_k$, we have ${\rm match}(G_k)={\rm match}(G_{k-1})+1=k$. On the other hand, using \cite[Lemma 2.1]{SF16} and the induction hypothesis, we have $$\ordm(G_k)\geq \ordm(G_{k-1})+1=k.$$Thus, $\ordm(G_k)=k$.
\end{example}

In \Cref{exmatchsell}, we showed that for each integer $k\geq 3$, there is a non-sequentially Cohen-Macaulay bipartite graph $G$ with ${\rm match}(G)=\ordm(G)=k$. The following proposition shows that we cannot expect such an example when $k\leq 2$.

\begin{proposition}\drlabel{propseqCM}
Let $G$ be a bipartite graph with ${\rm match}(G)=\ordm(G)\leq 2$. Then $G$ is a sequentially Cohen-Macaulay graph.    
\end{proposition}
\begin{proof}
By contradiction, suppose $G$ is not sequentially Cohen-Macaulay. It is known that any forest is sequentially Cohen-Macaulay (see for instance, \cite[Theorem 1.3]{VV08}). As a consequence, $G$ is not a forest and therefore, has a cycle $C$. Since ${\rm match}(G)\leq 2$, the length of $C$ is equal to four. Suppose $V(C)=\{w,x,y,z\}$  and $E(C)=\{wx,xy,yz,wz\}$. Since ${\rm match}(C)=2$ and $\ordm(C)=1$, we conclude that $G\neq C$. This implies that there is an edge say $wv$ connected to $C$, where $v$ is a vertex in $V(G)\setminus V(C)$. As $G$ is a bipartite graph, $v$ is not adjacent to the vertices $x$ and $z$. If there is a vertex $u\in V(G)\setminus V(C)$ such that $ux\in E(G)$ or $uz\in E(G)$, then the matching number of $G$ would be at least three, which is a contradiction. Thus, $\deg_G(x)=\deg_G(z)=2$. If $G$ has a vertex $t\in V(G)\setminus V(C)$ whose distance from $w$ or $y$ is at least two, then again, the matching number of $G$ would be at least three, which is a contradiction. Hence, $V(G)=N_G[w]\cup N_G[y]$. Since $G$ is a bipartite graph two distinct vertices belonging to $N_G(w)\cup N_G(y)$ can not be adjacent. Hence,
$V(G)=\{w,y\}\sqcup (V(G)\setminus \{w,y\})$ is the bipartition for the vertex set of $G$. Consequently, $G$ is a subgraph of 
$K_{2,m}$, for some integer $m\geq 3$. If $G=K_{2,m}$, then 
$\ordm(G)=1$, which is a contradiction. Thus, $G\neq K_{2,m}$. So, $G$ has a leaf, say $s$. Then the unique neighbor of $s$ is either $w$ or $y$. Without loss of generality, we may assume that $ws\in E(G)$. Then the graphs $G\setminus N_G[w]$ and $G\setminus N_G[s]$ are forests (as they do not contain the vertex $w$ and so, the cycle $C$). Therefore, using \cite[Theorem 1.3]{VV08}, the graphs $G\setminus N_G[w]$ and $G\setminus N_G[s]$ are sequentially Cohen-Macaulay. Hence, \cite[Corollary 3.11]{VV08} implies that $G$ is a sequentially Cohen-Macaulay graph, which is a contradiction. 
\end{proof}

In the following result, we give a characterization of bipartite graphs with  $\ordm(G)=\indm(G)=1$. A classification of bipartite graphs with ord-match and ind-match being equal to an arbitrary positive integer strictly bigger than 1 is in \Cref{thmordindclass}.

\begin{theorem}\drlabel{thmindred1}
Let $G$ be a bipartite graph. Then $\ordm(G)=\indm(G)=1$ if and only if $G$ is a complete bipartite graph.
\end{theorem}
\proof
If $G = K_{m,n}$ for some $m,n\geq 1$, then clearly $\indm(G)=1=\ordm(G)$. Conversely suppose $G$ is not a complete bipartite graph. Write $V(G) = \{x_1,\ldots, x_m\}\sqcup \{y_1,\ldots,y_n\}$. Since $G$ is not complete bipartite, there exist $i, j$ such that $x_iy_j \notin E(G)$. By permuting the vertices, we may assume that $x_jy_j \in E(G)$. Choose an $r$ such that $x_iy_r \in E(G)$. Then $\{x_jy_j, x_iy_r\}$ is an ordered matching in $G$. Hence $\ordm(G) > 1$.
\qed

\begin{definition}\drlabel{defcI}
For every subset $I$ of $X$,
let $C_I$ be the set of all vertices in $Y$
that have an edge to all the vertices in $I$ and to none in $X \setminus I$.
Set $c_I = |C_I|$.
\end{definition}

Sometimes we shorten the notation and write in the subscripts not the set but its elements, without commas. For example, we abbreviate $c_{\{i,j\}}$ to $c_{ij}$. 
If $I, J$ are distinct subsets of $X$,
then $C_I$ and $C_J$ are disjoint.
In fact,
each $y \in Y$ belongs to exactly one of $C_I$.
Since $G$ does not have isolated vertices, we have
$C_\emptyset = \emptyset$.

In order to characterize bipartite graphs $G$ with $\indm(G)=\ordm(G)$, we need the following two propositions.

%For any edge $xy$ in $G$,
%the set $I$ of elements of $X$ with an edge to $y$ contains $x$
%and $y \in C_I$ and so $c_I > 0$.

%\proof
%Let $I$ be the set of all elements of $X$ with an edge to $y$.
%Then $x \in I$ and $y \in C_I$.
%\qed

\begin{proposition}\drlabel{propinduced}
Let $G$ be a bipartite graph and let $r$ be a positive integer.
Then the induced matching number of $G$ is at least~$r$
if and only if there exist subsets $J_1, J_2, \ldots, J_r$ of $X$ such that none of the $J_i$ is contained in the union of the others and such that $c_{J_1} c_{J_2} \cdots c_{J_r} > 0$.
\end{proposition}

\proof
Suppose that $\indm(G)$ is at least $r$.
Then there exist $a_1, \ldots, a_r \in X$ and $b_1, \ldots, b_r \in Y$ such that $a_1b_1, a_2b_2, \ldots, a_rb_r$ is an induced matching.
For each $i \in [r]$, let $J_i$ be the set of all $x \in X$
with an edge to $b_i$.
Then $a_i \in J_i$, $a_j \not \in J_i$ for all $j \not = i$,
and $b_i \in C_{J_i}$.
The conclusion follows for these $J_1, \ldots, J_r$.

Conversely,
suppose that there exist subsets $J_1, J_2, \ldots, J_r$ of $X$ such that none is contained in the union of the others and such that $c_{J_1} c_{J_2} \cdots c_{J_r} > 0$.
For each $i \in [r]$,
let $a_i \in J_i$ that is not in the union of the other $J_j$.
Since $c_{J_i} > 0$, there exists $b_i \in C_{J_i}$.
Then by the definition of these sets,
$a_1b_1, a_2b_2, \ldots, a_rb_r$ is an induced matching,
so that $\indm(G)$ is at least $r$.
\qed

\begin{proposition}\drlabel{propinducedordered}
Let $G$ be a bipartite graph and let $r$ be a positive integer. Then the following are equivalent:
\begin{enumerate}
    \item $\ordm(G) \geq r$. %Ordered matching number of $G$ is at least $r$.
    \item There exist subsets $J_1, J_2, \ldots, J_r$ of $X$
    such that $c_{J_1} c_{J_2} \cdots c_{J_r} > 0$
    and for all $i \in [r]$,
    $J_i$ is not contained in $J_1 \cup \cdots \cup J_{i-1}$.
    %and such $J_r \not = X$ when $r > 1$.
\end{enumerate}
Moreover, $\ordm(G) \leq r$
%ord-match of $G$ is at most $r$
if and only if for all possible sets $J_1, \ldots, J_r$ with the properties as in (2),
their union equals~$X$.
\end{proposition}

\proof
(1) $\Rightarrow$ (2):
Suppose that $\ordm(G)$ is at least $r$.
Then there exist $a_1, \ldots, a_r \in X$ and $b_1, \ldots, b_r \in Y$ such that 
$a_1b_1, a_2b_2, \ldots, a_rb_r$ is an ordered matching.
For each $i \in [r]$,
let $J_i$ be the set of all $x \in X$ with an edge to $b_i$.
Then $a_i \in J_i$, $a_{i+1}, \ldots, a_r \not \in J_i$,
and $b_i \in C_{J_i}$.
Thus (2) follows for these $J_1, \ldots, J_r$.

(2) $\Rightarrow$ (1):
This is trivial for $r = 1$, so we may assume that $r > 1$.
Let $J_1, J_2, \ldots, J_r$ be subsets of $X$
    such that for all $i \in [r]$,
    $J_i$ is not contained in $J_1 \cup \cdots \cup J_{i-1}$ and
    such that $c_{J_1} c_{J_2} \cdots c_{J_r} > 0$.
For each $i \in [r]$,
let $a_i \in J_i \setminus J_1 \cup \cdots \cup J_{i-1}$.
Since $c_{J_i} > 0$, there exists $b_i \in C_{J_i}$.
By definition of these sets,
$a_1b_1, a_2b_2, \ldots, a_rb_r$ is an ordered matching.
This proves (1).

If there exists $a \in X \setminus J_1 \cup \cdots \cup J_r$, 
then since $a$ is not an isolated vertex,
there is a vertex $b \in Y$ that is adjacent to $a$.
By definition of the sets $C_I$,
there is no edge between $a$ and $b_1, \ldots, b_r$.
Thus $a_1b_1, a_2b_2, \ldots, a_rb_r, ab$ is an ordered matching.
Thus $\ordm(G)$ is strictly bigger than~$r$.
Then by the equivalence of (1) and (2),
the union of $J_1, \ldots, J_r$ cannot be $X$.
This proves the last part.
\qed

Using Propositions \ref{propinduced} and \ref{propinducedordered}, we can classify all bipartite graphs for which ord-match and ind-match are equal.

\begin{theorem}\drlabel{thmordindclass} (Classification)
Let $G$ be a bipartite graph and let $r > 1$ be a positive integer.
    Let $J_1, \ldots, J_z$ be all the subsets $I$ of $X$ for which $c_I$ is positive.
    Then $G$ has ind-match and ord-match equal to~$r$ if and only if the following conditions are satisfied:
    \begin{enumerate}
        \item $z \ge r$.
        \item There exist distinct $j_1, \ldots, j_r \in [z]$ such that none of the $J_{j_i}$ is contained in the union of the remaining $J_{j_k}$.
        \item
For all $j_1, \ldots, j_r$ as in (2),
 $J_{j_1} \cup \cdots \cup J_{j_r} = X$.
    \end{enumerate}
\end{theorem}

\proof
First suppose that $\indm(G)$ and $\ordm(G)$ equal~$r$.
Then by \Cref{propinduced},
(1) and (2) hold.
\Cref{propinducedordered} implies (3).

Now suppose that the three conditions are satisfied.
Then by \Cref{propinduced}, induced matching number of $G$ is at least $r$
and $\ordm(G)$ is at most $r$.
But $\ordm(G)$ is greater than or equal to $\indm(G)$. 
So, $\indm(G)$ and $\ordm(G)$ are both equal to $r$.
\qed

We can say more in case $\ordm(G)$ and $\indm(G)$ are both equal to~2:

\begin{theorem}\drlabel{thmordindclasstwo} (Classification for r = 2)
Let $G$ be a bipartite graph.
    Let $J_1, \ldots, J_z$ be all the subsets $I$ of $X$ for which $c_I$ is positive.
    Then $G$ has ind-match and ord-match equal to~$2$ if and only if the following conditions are satisfied:
    \begin{enumerate}
        \item $z \ge 2$.
        \item There exist distinct $i, j \in [z]$ such that neither $J_i$ nor $J_j$ is contained in the other.
        \item
        %Any two distinct $i, j \in [z]$ can be relabeled so that after that relabeling,
        %$J_j$ is not contained in $J_i$ and $J_i \cup J_j = X$.
        For any two distinct $i, j \in [z]$, $J_i \cup J_j = X$.
    \end{enumerate}

Furthermore, if some $J_i$ and $J_j$ are disjoint,
then $z$ equals $2$ or $3$;
the latter exactly when the graph is connected and in this case $c_X > 0$.
\end{theorem}

\proof
Equivalence follows from \Cref{thmordindclass}. Note that part (3) here is equivalent to part (3) in \Cref{thmordindclass} because $J_i$ and $J_j$ are distinct sets.
 
Let $k \in [z] \setminus \{i,j\}$.
Then by condition~(3),
$J_k$ contains the complement of $J_i$ and $J_j$,
so that if $J_i$ and $J_j$ are disjoint,
then $J_k$ contains $X$.
This means that $J_k$ has to equal $X$.
Thus either $z = 3$ and $G$ is connected,
or else $z = 2$ and $G$ is not connected.
\qed

As a consequence of the above theorem, we obtain an explicit graph theoretic characterization of bipartite graphs $G$ with $\indm (G) =\ordm(G)=2$.

\begin{corollary}\drlabel{corindred2}
Let $G$ be a bipartite graph. Then $\indm (G) =\ordm(G)=2$ if and only if the bipartite complement $G^{bc}$ of $G$ is the disjoint union of complete bipartite graphs $H_1, \ldots, H_s$ with $s\geq 2$ such that at least two of the $H_i$ are not isolated vertices.
\end{corollary}

\begin{proof}
Let $J_1, \ldots, J_z$ be all the subsets $I$ of $X$ for which $c_I$ is positive. Then for any pair of distinct integers $i, j \in [z]$, we have $J_i \cup J_j = X$ if and only if for any choice of vertices $y\in C_{J_i}$ and $y'\in C_{J_j}$, the equality $N_{G^{bc}}(y)\cap N_{G^{bc}}(y')=\emptyset$ holds. Thus, Condition (3) of Theorem \ref{thmordindclasstwo} (and in fact, by Proposition \ref{propinducedordered}, the inequality $\ordm(G)\leq 2$) is equivalent to say that $G^{bc}$ is the disjoint union of complete bipartite graphs. Moreover, note that $\indm (G)\geq 2$ if and only if $\indm (G^{bc})\geq 2$. Since $G^{bc}$ is the disjoint union of complete bipartite graphs, we deduce at least two of these components are not isolated vertices.
\end{proof}

%{\color{red} Remove the next corollary from here and add to introduction}

%\begin{corollary}\drlabel{corregindm}
%Let $G$ be a bipartite graph as described in Theorem \ref{thmordindclass}. Then for every integer $s\geq 1$, we have$${\rm reg}(I(G)^s)=2s+\indm(G)-1.$$    
%\end{corollary}
%\begin{proof}
%It follows from \cite[Theorem 4.5]{selviha} and \cite[Theorem 3.9]{SY23} that for every integer $s\geq 1$,$$2s+\indm(G)-1\leq {\rm reg}(I(G)^s)\leq 2s+\ordm(G)-1.$$Thus, the assertion follows from Theorem \ref{thmordindclass}.
%\end{proof}
Let $G$ be a bipartite graph. It is known that if $G$ is either unmixed or sequentially Cohen-Macaulay, then ${\rm reg}(S/I(G))=\indm(G)$. In the following theorem, for every pair of integers $r,m$ with  $2\leq r\leq m$, we construct a bipartite graph $G_{r,m}$ which is neither sequentially Cohen-Macaulay nor unmixed, moreover, ${\rm reg}(S/I(G_{r,m}))=\indm(G_{r,m})=\ordm(G_{r,m})=r$ and 
$\minm(G_{r,m})=m$. Hence, the class of graphs we study in this paper is not contained in two general classes of bipartite graphs for which the regularity of edge ideals is known. Furthermore, Theorem \ref{thmregordmatch} shows that the family of graphs $G$ with $\indm(G)=\ordm(G)$ is far from the class of graphs considered in \cite{HHKT16}.

\begin{theorem}\drlabel{thmregordmatch}
Let $2\leq r\leq m$  be positive integers. Then there is a connected  bipartite graph $G_{r,m}$ such that
\begin{enumerate}
\item ${\rm reg}(S/I(G_{r,m}))=\indm(G_{r,m})=\ordm(G_{r,m})=r$ and 
$\minm(G_{r,m})=m$.
\item $G_{r,m}$ does not have any leaf (and hence $G_{r,m}$ is not a sequentially Cohen-Macaulay graph).
\item $G_{r,m}$ is not an unmixed graph.
\end{enumerate}
%with$${\rm reg}(S/I(G_{r,m}))=\indm(G_{r,m})=\ordm(G_{r,m})=r \ \ {\rm and} \ {\rm match}(G_{r,m})=m.$$
\end{theorem}

\begin{proof}
Set $X=\{x_1, \ldots, x_{m+1}\}$ and $Y=\{y_1, \ldots, y_{m+1}\}$. Let $G_{r,m}$ be the bipartite graph with vertex set $V(G_{r,m})=X\sqcup Y$ and edge set$$E(G_{r,m})=\bigcup_{1\leq j\leq r-1}\{x_1y_j, x_{j+1}y_j\}\cup\bigcup_{r\leq j\leq m}\{x_1y_j, x_iy_j\mid r+1\leq i\leq m+1\}\cup\{x_iy_{m+1}\mid 1\leq i\leq m+1\}.$$
For the convenience of the reader, we illustrate the graph $G_{r,m}$ with a sample picture below:

\begin{figure}[h]
    \centering
\begin{tikzpicture}[line cap=round,line join=round,>=triangle 45,x=1.0cm,y=1.0cm]
%\clip(-0.5576162659154316,-5.684938533002716) rectangle (19.662410760691284,7.6055011410363385);
\draw (1.,3.)-- (1.,1.);
\draw (1.,3.)-- (3.,1.);
\draw (1.,3.)-- (5.,1.);
\draw (1.,3.)-- (7.,1.);
\draw (1.,3.)-- (9.,1.);
\draw (3.,3.)-- (9.,1.);
\draw (5.,3.)-- (9.,1.);
\draw (7.,3.)-- (9.,1.);
\draw (9.,3.)-- (9.,1.);
\draw (3.,3.)-- (1.,1.);
\draw (5.,3.)-- (3.,1.);
\draw (5.,3.)-- (5.,1.);
\draw (5.,3.)-- (7.,1.);
\draw (7.,3.)-- (3.,1.);
\draw (7.,3.)-- (5.,1.);
\draw (7.,3.)-- (7.,1.);
\draw (9.,3.)-- (3.,1.);
\draw (9.,3.)-- (5.,1.);
\draw (9.,3.)-- (7.,1.);
\begin{scriptsize}
\draw [fill=black] (1.,3.) circle (1.5pt);
\draw[color=black] (1,3.2) node {$x_1$};
\draw [fill=black] (3.,3.) circle (1.5pt);
\draw[color=black] (3.,3.2) node {$x_2$};
\draw [fill=black] (5.,3.) circle (1.5pt);
\draw[color=black] (5.,3.2) node {$x_3$};
\draw [fill=black] (7.,3.) circle (1.5pt);
\draw[color=black] (7.,3.2) node {$x_4$};
\draw [fill=black] (9.,3.) circle (1.5pt);
\draw[color=black] (9.,3.2) node {$x_5$};
\draw [fill=black] (1.,1.) circle (1.5pt);
\draw[color=black] (1.,0.8) node {$y_1$};
\draw [fill=black] (3.,1.) circle (1.5pt);
\draw[color=black] (3.,0.8) node {$y_2$};
\draw [fill=black] (5.,1.) circle (1.5pt);
\draw[color=black] (5.,0.8) node {$y_3$};
\draw [fill=black] (7.,1.) circle (1.5pt);
\draw[color=black] (7.,0.8) node {$y_4$};
\draw [fill=black] (9.,1.) circle (1.5pt);
\draw[color=black] (9.,0.8) node {$y_5$};
\end{scriptsize}
\end{tikzpicture}
    \caption{$G_{2,4}$}
    \drlabel{fig:g24}
\end{figure}

It follows directly from the definition that \begin{equation*}
N_{G_{r,m}}(y_j) = \left\{
        \begin{array}{ll}
            \{x_1,x_{j+1}\}, & \quad 1\leq j\leq r-1; \\
            \{x_{r+1}, \ldots, x_{m+1}\}\cup\{x_1\}, & \quad r\leq j\leq m;\\
            X, & \quad j=m+1. 
        \end{array}
    \right.    
\end{equation*}
Clearly, $G_{r,m}$ is a connected bipartite graph that does not have any leaf. Therefore using \cite[Corollary 2.10]{VV08}, it is not a sequentially Cohen-Macaulay graph. Moreover, it is easy to check that the set $\{x_2, \ldots, x_r, y_r, \ldots, y_m\}$ is a maximal independent set of $G_{r,m}$ with cardinality $m<|X|$. Thus, $G_{r,m}$ is not an unmixed graph. Hence, we only need to verify condition (i).

Set
$$J_1=N_{G_{r,m}}(y_1), J_2=N_{G_{r,m}}(y_2), \ldots, J_r=N_{G_{r,m}}(y_r), J_{r+1}=X.$$Note that for a subset $J\subseteq X$, we have $c_J> 0$ if and only if $J\in \{J_1, \ldots, J_{r+1}\}$. Then it follows from Theorem \ref{thmordindclass} that $${\rm reg}(S/I(G_{r,m}))=\indm(G_{r,m})=\ordm(G_{r,m})=r.$$Now, we compute 
$\minm(G_{r,m})$. It is obvious that $$\{y_2x_3, y_3x_4,\ldots, y_{m-1}x_m, y_mx_1, y_{m+1}x_2\}$$ is a maximal matching in $G_{r,m}$ of size $m$. Hence, $\minm(G_{r,m})\leq m$. Suppose $M$ is a maximal matching in $G_{r,m}$ with $|M|\leq m-1$. Thus, there are two distinct vertices $y_i, y_j\in Y\setminus V(M)$. Without loss of generality, we may assume that $i<j$. We consider the following cases.

{\bf Case 1.} Suppose $i,j\leq r-1$, then at least one of the vertices $x_{i+1}$ and $x_{j+1}$ does not belong to $V(M)$, as otherwise the edges $y_{m+1}x_{i+1}$ and $y_{m+1}x_{j+1}$ belong to $M$, which contradicts the definition of a matching. For example, suppose $x_{i+1}\notin V(M)$. Then $M\cup \{y_ix_{i+1}\}$ is a matching in $G_{r,m}$. Thus, $M$ is not a maximal matching.

{\bf Case 2.} Suppose $1\leq i\leq r-1$ and $r\leq j\leq m$. Since $M$ is a maximal matching in $G_{r,m}$, we deduce that $N_{G_{r,m}}(y_i)\cup N_{G_{r,m}}(y_j)\subseteq V(M)$. This implies that$$\{x_1, x_{i+1}, x_{r+1}, \ldots, x_{m+1}\}\subseteq V(M).$$Since $x_{i+1}\in V(M)$ and $y_i\notin V(M)$, the edge $y_{m+1}x_{i+1}$ must belong to $M$. Then since $x_{r+1}, \ldots, x_{m+1}\in V(M)$, we conclude that $y_r, \ldots, y_m\in V(M)$. Moreover, according to Case 1, we may assume that $\{y_1, \ldots, y_{r-1}\}\setminus \{y_i\}\subseteq V(M)$. Therefore, $Y\setminus \{y_i\}\subseteq V(M)$. So, $|M|\geq m$, which is a contradiction.

{\bf Case 3.} Suppose $r\leq i,j\leq m$. Since $M$ is a maximal matching, we must have $N_{G_{r,m}}(y_i)\cup N_{G_{r,m}}(y_j)\subseteq V(M)$. In particular, $\{x_{r+1}, \ldots, x_{m+1}\}\subseteq V(M)$. So, in $G_{r,m}$, there are at least $m-r+1$ vertices other that $y_i$ and $y_j$, which are adjacent to at least one of the vertices $x_{r+1}, \ldots, x_{m+1}$. But this is not the case according to the construction of $G_{r,m}$.

{\bf Case 4.} Suppose $j=m+1$. Since $M$ is a maximal matching, $N_{G_{r,m}}(y_j)\subseteq V(M)$. Therefore, $X\subseteq V(M)$. This implies that $|M|\geq m+1$, which is a contradiction.
\end{proof}

\section{Set-up of $k$-sequences and J-sets} \label{sec3}

The notation in this section sets the stage for the counting
of graphs in the next section.

\begin{definition}\drlabel{defk-sequence}
A {\bf $\bf k$-sequence} is a strictly decreasing sequence of integers
$\{k_0, k_1, \ldots, k_z$\} with $z \ge 2$, $k_z = 0$,
and for all $l = 2, 3, \ldots, z$,
$k_{l+1}\ge 2 k_l - k_{l-1}$.
\end{definition}

Observe that if a $k$-sequence has at least four terms,
then the truncation of the sequence by removing the first term is also a $k$-sequence.
A truncated $k$-sequence is a truncation of infinitely many $k$-sequences.

%\begin{example}\drlabel{exkseq}
\begin{examples}\drlabel{exkseq}
\rm
When $k_0 = 2$,
the only possible $k$-sequence is $\{2,1,0\}$.

%When $k_0 = 3$,
%the sequence $\{3,2,0\}$ is not a $k$-sequence because $0 \not \ge 2 \cdot 2 - 3$.
The only possible $k$-sequences with $k_0 = 3$ are
$\{3,2,1,0\}$ and $\{3,1,0\}$.

When $k_0 = 4$,
there are four $k$-sequences:
$\{4,3,2,1,0\}$; $\{4,2,1,0\}$; $\{4,2,0\}$ and  $\{4,1,0\}$.

When $k_0 = 5$,
there are six $k$-sequences:
$\{5,4,3,2,1,0\}$; $\{5,3,2,1,0\}$; $\{5,3,1,0\}$; $\{5,2,1,0\}$; $\{5,2,0\}$ and $\{5,1,0$\}.
%\end{example}
\end{examples}

\begin{definition}\drlabel{defsummation}
For any non-negative integer $n$ and any $k$-sequence $\{k_0, k_1, \ldots, k_z\}$ with $z \ge 3$,
set
$$
D(n; \{k_0, \ldots, k_z\}) = 
\begin{cases}
    \displaystyle \sum_{i=1}^{n-1} \sum_{j=1}^{n-i} C(n-i,j; k_1, \ldots, k_z),
        & \hbox{if $k_2 > 2 k_1 - k_0$}; \\
    \displaystyle \sum_{i=1}^{n-1} \sum_{j=1}^{\min\{n-i,i\}} C(n-i,j; k_1, \ldots, k_z),
        & \hbox{if $k_2 = 2 k_1 - k_0$}, \\
\end{cases}
$$
where $C$ is defined recursively as follows:
$$
C(n,i; k_1, \ldots, k_z) =
\begin{cases}
    n-i+1, & \hbox{if $z = 3$ and $k_3 > 2 k_2 - k_1$}; \\
    \min\{n-i,i\}+1, & \hbox{if $z = 3$ and $k_3 = 2 k_2 - k_1$}; \\
    \displaystyle \sum_{j=0}^{n-i} C(n-i, j; k_2, \ldots, k_z), & \hbox{if $z > 3$ and $k_3 > 2 k_2 - k_1$}; \\
    \displaystyle \sum_{j=0}^{\min\{n-i,i\}} C(n-i, j; k_2, \ldots, k_z), 
        & \hbox{if $z > 3$ and $k_3 = 2 k_2 - k_1$}. \\
        %& \hbox{if $z \ge 3$ and $k_3 = 2 k_2 - k_1$}. \\
    \end{cases}
$$
\end{definition}

Here is a full expansion of $D$:
$$
D(n; \{k_0, \ldots, k_z\})
=
\sum_{i_1 = 1}^{n-1} \sum_{i_2 = 1}^{u_2}
\sum_{i_3 = 0}^{u_3} \sum_{i_4 = 0}^{u_4} \ldots \sum_{i_z = 0}^{u_z} 
%\sum_{i_{z+1} = 0}^{n-i_1 - \cdots - i_z} 
1,
$$
where the upper bounds $u_l$ equal either $n-i_1-\cdots-i_{l-1}$
or $\min\{n-i_1-\cdots-i_{l-1}, i_{l-1}\}$, depending on whether strict inequality or equality hold in $k_{l+1} \ge 2 k_l - k_{l-1}$.
The sum does not depend on the specific values in the $k$-sequence,
but only on whether $k_{l+1}$ is strictly greater than or equal to $2k_l - k_{l-1}$.
Note that the counters $i_1$ and $i_2$ start with $1$, but all other counters $i_l$ start with~$0$.

\begin{example}
We work out $C$ explicitly on $k$-sequences from
Examples~\ref{exkseq} %Cref does not work
with $z \ge 3$.
\begin{align*}
C(n,i;2,1,0) &= \min\{n-i,i\}+1, \\
C(n,i;3,1,0) &= n-i+1, \\
C(n,i;3,2,1,0) &= \sum_{j=0}^{\min\{n-i,i\}} C(n-i, j; 2,1,0)
= \sum_{j=0}^{\min\{n-i,i\}} (\min\{n-i-j,j\}+1), \\
C(n,i;4,3,2,1,0) &= \sum_{j=0}^{\min\{n-i,i\}} C(n-i, j; 3,2,1,0) \\
&= \sum_{j=0}^{\min\{n-i,i\}} \sum_{k=0}^{\min\{n-i-j,j\}} (\min\{n-i-j-k,k\}+1).
\\
\end{align*}
\end{example}

\begin{definition}\drlabel{defsummationinclexcl}
Let $K$ be a finite set of $k$-sequences,
each $k$-sequence in $K$ having at least four terms (so $z \ge 3$ for each sequence in $K$). 
Let $W$ be the ordered set from largest to smallest of all non-negative integers 
that appear in all sequences in $K$.
So $0 \in W$.
For any positive integer $l \le |W|$ and any $\underline k \in K$,
let $o(l,\underline k)$ be the subscript index of the $l$th element of $W$ in $\underline k$.
%and let $D'(n; \underline k)$ be the sum of those summands in $D(n; \underline k)$
%for which the upper bound in the $j$th summation is the same as the lower bound
%exactly when $k_j \not \in W$.
For any integer $n \ge 2$ we define $D_K(n)$ to be
$0$ if the first three terms in the sequences in $K$ are not the same,
and otherwise
$D_K(n)$ is the sum of the common summands in the $D$-sums (from \Cref{defsummation})
coming from the $W$-terms for all sequences in~$K$.
Explicitly,
$$
D_K(n) =
\sum_{i_1 = 1}^{n-1} \sum_{i_2 = 1}^{v_2}
\sum_{i_3 = 0}^{v_3} \sum_{i_4 = 0}^{v_4} \ldots \sum_{i_{|W|} = 0}^{v_{|W|}} 1,
$$
where for $l = 2, \ldots, |W|$,
$v_l$ is the minimum of all the upper bounds in $D(n; \underline k)$ in the $o(l, \underline k)$th summation as $\underline k$ varies over all sequences in $K$.
These upper bounds depend also on the summation indices not appearing in $W$,
and those hidden indices are all set to their common hidden value $0$.
\end{definition}

For example,
$D_K(n) = 0$ for the set $K$ consisting of two $k$-sequences $\{6,3,2,1,0\}$ and $\{6,3,1,0\}$ because their third terms are not the same.

For a more interesting example,
let $K$ consist of $\underline k = \{6,4,2,1,0\}$ and $\underline l = \{6,4,2,0\}$.
Then $W = \{6,4,2,0\}$,
$o(6,\underline k) = o(6,\underline l) = 0$,
$o(4,\underline k) = o(4,\underline l) = 1$,
$o(2,\underline k) = o(2,\underline l) = 2$,
$o(0,\underline k) = 4$, $o(0,\underline l) = 3$.
We write out explicitly
$$
D(n; \{6,4,2,1,0\}) = \sum_{i_1=1}^{n-1} 
\sum_{i_2 = 1}^{\min\{n-i_1,i_1\}}
\sum_{i_3 = 0}^{n-i_1-i_2}
\sum_{i_4 = 0}^{\min\{n-i_1-i_2-i_3,i_3\}} 1
$$
and
$$
D(n; \{6,4,2,0\}) = \sum_{i_1=1}^{n-1} 
\sum_{i_2 = 1}^{\min\{n-i_1,i_1\}}
\sum_{i_3 = 0}^{\min\{n-i_1-i_2,i_2\}} 1
= \sum_{i_1=1}^{n-1} 
\sum_{i_2 = 1}^{\min\{n-i_1,i_1\}}
\sum_{i_3 = 0} ^0
\sum_{i_4 = 0}^{\min\{n-i_1-i_2,i_2\}} 1,
$$
where in the last line we inserted the trivial summation $\sum_{i_3 = 0}^0$ to correspond to
the entry $1$ in $\underline k$ that does not exist in $\underline l$,
and we correspondingly renamed the old $i_3$ as new $i_4$.
From these common rewritings of the two summations we read off
$$
D_K(n) = \sum_{i_1 = 1}^{n-1} \sum_{i_2 = 1}^{\min\{n-i_1,i_1\}}
\sum_{i_3 = 0}^0 \sum_{i_4 = 0}^{\min\{n-i_1 - i_2 - i_3,i_3\}}  
1
 = \sum_{i_1 = 1}^{n-1} \sum_{i_2 = 1}^{\min\{n-i_1,i_1\}} 1.
$$
Incidentally,
this sum simplifies to 
$\frac{1}{2} \lceil n/2 \rceil(\lceil n/2 \rceil-1) + \frac{1}{2} \lfloor n/2 \rfloor (\lfloor n/2 \rfloor +1)$.

The non-trivial part of the last simplification involves exactly two non-trivial summations,
and that is due to the two sequences having exactly two J-sets in common;
J-sets are introduced next.

We will use J-sets in the counting of {\bf connected bipartite} graphs
with ind-match and ord-match both equal to~$2$.
%The idea is that for each such graph we label sets $J_1, J_2, \ldots$ that play the role of sets in \Cref{defcI} for which $c_{J_i} > 0$.

\begin{definition}\drlabel{defJsets}
We fix a connected bipartite graph $G$ on $X \sqcup Y$
with $\ordm(G)=\indm(G)=2$.
{\bf J-sets} are elements of
a list $J_1, J_2, \ldots, J_z$ of non-empty proper subsets of~$X$
such that
\begin{enumerate}
\item
$z \ge 2$.
\item
$|J_1| \le |J_2| \le \cdots \le |J_z|$.
\item
$J_i \cup J_j = X$ for all distinct $i, j \in [z]$.
\item
If $I$ is a proper subset of $X$ such that $c_I > 0$,
then $I = J_l$ for some $l \in [z]$.
\item
$c_{J_1} c_{J_2} > 0$.
\item
$J_1 \cap \cdots \cap J_z = \emptyset$.
\end{enumerate}
%(This list may contain also some special subsets $I$ for which $c_I = 0$.
%Unlike in \Cref{thmordindclass} and \Cref{thmordindclasstwo}, here the sets are proper.)
%We require that all proper subsets $I$ of~$X$ for which $c_I > 0$ for a given connected bipartite graph on $X \sqcup Y$ appear on this list (see \Cref{defcI}), but this list of subsets may contain also some special subsets $I$ for which $c_I = 0$.

For $l = 0, \ldots, z$ we set $k_l = |J_1 \cap \cdots \cap J_l|$.
\end{definition}

\begin{remarks} \rm
These remarks are all about \Cref{defJsets}.

\begin{enumerate}
\item
The assumption $\ordm(G)=\indm(G)=2$ implies that $G$ is not complete bipartite.

\item
By \Cref{thmordindclasstwo},
$I \cup J = X$
for any two distinct subsets $I$ and $J$ of $X$ with $c_I c_J > 0$.
Thus a possible list $J_1, \ldots, J_z$ satisfying conditions (1)--(5)
consists of exactly those proper subsets $I$ of $X$ for which $c_I > 0$.

\item
Let's first look at those lists for which (1)--(5) hold.
So $k_0 = m$. % and $k_z = 0$.
By renaming,
without loss of generality $J_1 = \{1,2,\ldots,k_1\}$.
Since $J_2$ contains exactly $k_2$ elements of $J_1$
and all elements of $X \setminus J_1$,
without loss of generality
$J_2 = \{1,2,\ldots, k_2\} \cup \{k_1+1, k_1+2, \ldots, m\}$. 
Since $J_2$ is a proper subset of~$X$,
it follows that $k_2 < k_1$.
If $z \ge 3$,
then $J_3$ must contain $\{k_2+1, k_2+2, \ldots, m\}$
and also exactly $k_3$ elements of $J_1 \cap J_2 = \{1, 2, \ldots, k_2\}$.
So without loss of generality
$J_3 = \{1,2,\ldots, k_3\} \cup \{k_2+1, k_2+2, \ldots, m\}$.
Since $J_3$ is a proper subset,
necessarily $k_3 < k_2$.
Similarly,
for all $l \in [z]$,
$J_l = \{1,2,\ldots, k_l\} \cup \{k_{l-1}+1, k_{l-1}+2, \ldots, m\}$
and $k_0 = m > k_1 > k_2 > \cdots > k_z$.

\item
If we are not assuming condition~(6),
i.e., if $k_z > 0$ (and $z \ge 2$),
then we can add to the list also the proper subset
$J_{z+1} = \{k_z + 1, k_z + 2, \ldots, m\}$
and we increase~$z$ by~$1$,
to get that for the enlarged list, $k_{z+1} = 0$.
Note that for this new set,
$c_{J_{z+1}} = 0$.

\item
There may be several $l$ for which $c_{J_l} = 0$.

\item
Note that $|J_l| = m - k_{l-1} + k_l$.
The condition $|J_1| \le |J_2| \le \cdots \le |J_z| < |X| = m$
is equivalent to the condition
$k_{l+1} \ge 2 k_l - k_{l-1}$ for all $l = 1, 2, \ldots, z-1$.
Thus $k_0, \ldots, k_z = 0$ is a $k$-sequence (see \Cref{defk-sequence}).
In addition, for any $l$,
$|J_l| < |J_{l+1}|$ is equivalent to $k_{l+1} > 2 k_l - k_{l-1}$.

\item
Thus a J-set list determines a $k$-sequence,
and vice versa,
a $k$-sequence defines up to relabeling a J-set list.

%The sets $J_1$ and $J_2$ in the definition contribute edges to $G$ that will make ind-match and ord-match equal to~$2$, and the remaining $J_3, \ldots, J_z$ may only potentially contribute an edge.

\item
In the case that $J_1$ and $J_2$ are disjoint,
then any $J_3$ would have to contain the complements of $J_1$ and $J_2$,
so $J_3$ would not be proper.
Thus in this case necessarily $z = 2$,
and by \Cref{thmordindclasstwo}, $c_X$ is positive.
% and otherwise $c_X$ is only potentially positive.
\end{enumerate}
\end{remarks}

\begin{examples}\rm
We present J-sets for the k-sequences in
% \Cref{exkseq}.
Examples~\ref{exkseq}. %Cref does not work

The $k$-sequence $\{2,1,0\}$ corresponds to the sets $J_1 = \{1\}$, $J_2 = \{2\}$.

The $k$-sequence $\{3,2,1,0\}$ corresponds to the sets
$J_1 = \{1,2\}$, $J_2 = \{1, 3\}$, $J_3 = \{2,3\}$.
The $k$-sequence $\{3,1,0\}$ corresponds to the sets
$J_1 = \{1\}$, $J_2 = \{2, 3\}$.

The $k$-sequence $\{4,3,2,1,0\}$ corresponds to the sets
$J_1 = \{1,2,3\}$, $J_2 = \{1,2,4\}$, $J_3 = \{1,3,4\}$, $J_4 = \{2,3,4\}$;
the $k$-sequence $\{4,2,1,0\}$ corresponds to the sets
$J_1 = \{1,2\}$, $J_2 = \{1,3,4\}$, $J_3 = \{2,3,4\}$;
the $k$-sequence $\{4,2,0\}$ corresponds to the sets
$J_1 = \{1,2\}$, $J_2 = \{3,4\}$;
and the $k$-sequence $\{4,1,0\}$
corresponds to the sets $J_1 = \{1\}$, $J_2 = \{2,3,4\}$.

%When $m = 5$,
%the $k$-sequence $5,4,3,2,1,0$ corresponds to the sets $J_1 = \{1,2,3,4\}$, $J_2 = \{1,2,3,5\}$, $J_3 = \{1,2,4,5\}$, $J_4 = \{1,3,4,5\}$, $J_5 = \{2,3,4,5\}$;
%the $k$-sequence $5,3,2,1,0$ corresponds to the sets $J_1 = \{1,2,3\}$, $J_2 = \{1,2,4,5\}$, $J_3 = \{1,3,4,5\}$, $J_4 = \{2,3,4,5\}$;
%the $k$-sequence $5,3,1,0$ corresponds to the sets $J_1 = \{1,2,3\}$, $J_2 = \{1,4,5\}$, $J_3 = \{2,4,5\}$;
%the $k$-sequence $5,2,1,0$ corresponds to the sets $J_1 = \{1,2\}$, $J_2 = \{1,3,4,5\}$, $J_3 = \{2,3,4,5\}$;
%the $k$-sequence $5,2,0$ corresponds to the sets $J_1 = \{1,2\}$, $J_2 = \{3,4,5\}$;
%and the $k$-sequence $5,1,0$ corresponds to the sets $J_1 = \{1\}$, $J_2 = \{2,3,4,5\}$.

The $k$-sequence $\{6,4,2,1,0\}$,
corresponds to the sets
$J_1 = \{1,2,3,4\}$, $J_2 = \{1,2,5,6\}$, $J_3 = \{1,3,4,5,6\}$ and $J_4 = \{2,3,4,5,6\}$,
and the $k$-sequence $\{6,4,2,0\}$,
corresponds to
$J_1 = \{1,2,3,4\}$, $J_2 = \{1,2,5,6\}$, $J_3 = \{3,4,5,6\}$.
So the last two $k$-sequences have exactly two J-sets in common,
which explains why $D_K$ computed after \Cref{defsummationinclexcl}
has only two summations, all other indices varying trivially.
\end{examples}

%\section{Bipartite graphs for which induced and ordered matching equal~$2$}\label{sec:count}
\section{Bipartite graphs with induced and ordered matching~$2$}\label{sec:count}

The aim of this section is to count the number of connected bipartite graphs $G$ whose ord-match and ind-match are equal to $2$. For the rest of the section, let $G$ denote a  connected bipartite graph. Suppose $X \sqcup Y$ is the bipartition for the vertex set of $G$. Set $m := |X|$ and $n := |Y|$. So $G$ is a connected spanning subgraph of $K_{m,n}$. Without loss of generality we assume that $m \leq n$. Moreover, suppose that $\indm(G) = \ordm(G)=2$. As the main result of this section, in Theorem \ref{thmall}, we provide an inclusion-exclusion type formula for the number of such graphs. As a consequence, in Corollary \ref{corclosed}, we obtain a closed formula for the number of these graphs when $m\leq 4$.

We first count a special type of connected bipartite graphs whose ind-match and ord-match are equal to $2$.

\begin{proposition}\drlabel{propdisjoint}
Assume that $I$ is a proper subset of $X$ and that $n=|Y|\geq 3$.
The number of connected spanning subgraphs $G$ of $K_{m,n}$ with $\indm(G)=\ordm(G)=2$ %requires $n \ge 3$,
for which $c_I c_{X \setminus I} > 0$
equals $\binom{n-1}{2}$ if $|I| \not = n/2$
and $\left\lfloor \frac{n}{2} \right\rfloor
\left(
\left\lceil \frac{n}{2} \right\rceil -1
\right)$ otherwise.
\end{proposition}

\proof
By \Cref{thmordindclasstwo},
$I$, $X \setminus I$ and $X$ are the only subsets $K$ of $X$ for which 
$c_K$ is positive.

When $|I| \not = |X \setminus I|$, then the number of such graphs equals
$$
\sum_{c_I=1}^{n-2} \sum_{c_{X \setminus I}=1}^{n-c_I-1} 1 
= \sum_{i=1}^{n-2} \sum_{j=1}^{n-i-1} 1
= \sum_{i = 1}^{n-2} (n-i-1)
= \binom{n-1}{2},
$$
and if $|I| = |X \setminus I|$, then to not count unlabeled graphs twice,
the number of such graphs equals
$$
\sum_{c_I=1}^{n-2} \sum_{c_{X \setminus I}=1}^{\min\{c_I, n-c_I-1\}} 1
= \sum_{i=1}^{n-2} \sum_{j=1}^{\min\{i, n-i-1\}} 1,
$$
which by \Cref{propsumdisjoint} equals
$\left\lfloor \frac{n}{2} \right\rfloor
\left(
\left\lceil \frac{n}{2} \right\rceil -1
\right)$.
\qed

The next proposition shows that the number of connected bipartite graphs $G$ with $\indm(G) = \ordm(G)=2$ is closely related to the notion of $k$-sequences defined in Section \ref{sec3}.

\begin{proposition}\drlabel{propcountforseq}
Let $z \ge 3$ and $k_0 = m, k_1, \ldots, k_z = 0$ a $k$-sequence
with corresponding sets $J_1, \ldots, J_z \subsetneq X = \{1,\ldots, m\}$ as in \Cref{defJsets}.
Then $D(n; \{m, k_1, \ldots, k_z\})$
equals the number of connected bipartite graphs with edges from vertices in $X$ to vertices in the set with~$n$ elements
for which ord-match and ind-match are both equal to~2,
and for which $c_{J_1} c_{J_2} > 0$
and $c_I = 0$ for all proper subsets $I$ of $X$ different from $J_1, \ldots, J_z$.
\end{proposition}

\proof
The count is of unlabelled sets of vertices.
With the set-up as in \Cref{defJsets},
we are counting the number of $(z+1)$-tuples
$(c_{J_1}, \ldots, c_{J_z}, c_X)$ for which $c_{J_1}, c_{J_2} > 0$,
$(\sum_{l=1}^z c_{J_l}) + c_X = |Y| = n$,
and $c_I = 0$ for all proper subsets $I$ of $X$ other than $J_1, \ldots, J_z$.
In order to not count identical unlabeled graphs twice,
we need a further restriction:
whenever $J_{l}, J_{l+1}, \ldots, J_{l+k}$ have the same number of elements,
then we may assume that $c_{J_l} \ge c_{J_{l+1}} \ge \cdots \ge c_{J_{l+k}}$.
If $J_l$ and $J_{l+1}$ have a different number of elements,
then there is no restriction on the comparison of sizes of $c_{J_l}$ and $c_{J_{l+1}}$.
Thus the count equals precisely $D(n; \{m, k_1, \ldots, k_z\})$ from \Cref{defsummation}.
\qed

In order to count all connected bipartite graphs with ord-match and ind-match both being equal to~$2$, we thus need a list of all $k$-sequences
(or equivalently, of corresponding J-sets for one of the two sets of vertices),
count the graphs for each of the $k$-sequences,
and then use inclusion-exclusion to remove multiple counts of some of the graphs.
This will be done in the next theorem.

%\begin{theorem}\drlabel{thmall}
%Let $N(m,n)$ be $\left\lfloor{\frac{m-1}{2}}\right\rfloor {n-1 \choose 2}$
%plus $E = \left\lfloor \frac{n}{2} \right\rfloor \left( \left\lceil \frac{n}{2} \right\rceil-1 \right)$
%if $m$ is even,
%plus the inclusion-exclusion sum $\sum_K (-1)^{|K|-1} D_K(n)$,
%where $K$ varies over all non-empty subsets of the set of all $k$-sequences starting with $k_0 = m$ and with at least $4$ terms, and where $D_K$ is from \Cref{defsummationinclexcl}.
%
%If $m < n$,
%then $N(m,n)$ equals the number of connected bipartite graphs
%with edges between sets of $m$ and $n$ vertices
%with both ord-match and ind-match being equal to~$2$.
%
%When $m = n$,
%then $N(m,n)$ counts the graphs as before but treating the two sets of vertices as labeled.
%\end{theorem}
%
%\textcolor{red}{OLD ABOVE, NEW BELOW}
We set some notation needed for the next theorem. 
\begin{notation}
For $m \le n$, define
$$
N(m,n) =
\left\lfloor{\frac{m-1}{2}}\right\rfloor {n-1 \choose 2}
+ (\delta_{m, \hbox{\tiny even}}) %\left(m+1- 2\left\lceil \frac{m}{2}\right\rceil\right) 
\left\lfloor \frac{n}{2} \right\rfloor
\left(
\left\lceil \frac{n}{2} \right\rceil-1
\right)
+\sum_K (-1)^{|K|-1} D_K(n),
$$
where $K$ varies over all non-empty subsets of the set of all $k$-sequences starting with $k_0 = m$ and with at least $4$ terms, where $D_K$ is from \Cref{defsummationinclexcl},
and $\delta_{m, \hbox{\tiny even}}$ is $1$ if $m$ is even and $0$ otherwise.
\end{notation}
\begin{theorem}\drlabel{thmall}
%(Note that the middle summand is $0$ if $m$ is odd and equals $\left\lfloor \frac{n}{2} \right\rfloor\left( \left\lceil \frac{n}{2} \right\rceil-1 \right)$ when $m$ is even.)
If $m < n$,
then $N(m,n)$ equals the number of connected non-isomorphic spanning subgraphs of $K_{m,n}$
%with edges between sets of $m$ and $n$ vertices
with both ord-match and ind-match being equal to~$2$.

When $m = n$,
then $N(m,n)$ counts the graphs as before with vertices unlabeled but treating the two sets of vertices as labeled.
\end{theorem}

\proof
Every graph under discussion corresponds to a unique collection of proper subsets $J_1, J_2, \ldots, J_z$
with non-decreasing cardinalities
and their corresponding $k$-sequence
$m = k_0, k_1, \ldots, k_z = 0$.

We first count the cases with $z = 2$,
i.e., when $J_1$ and $J_2$ are disjoint.
Then \Cref{propdisjoint} applies in each case.
Since $|J_1| \le |J_2|$,
necessarily $|J_1| \le m/2$.
The $\left\lfloor {m-1 \over 2} \right\rfloor$ cases in which $|J_1| < |J_2|$
each give ${n-1 \choose 2}$ graphs,
and when $m$ is even,
the case
$|J_1| = |J_2|$ gives
$\left\lfloor n \over 2 \right\rfloor
\left(
\left\lceil n \over 2 \right\rceil-1
\right)$
graphs.
This accounts for the first two summands in $N(m,n)$.

The third summand comes from counting non-disjoint J-subsets
as in \Cref{propcountforseq}.
\qed

In the following corollary, we obtain a closed formula for the number of connected bipartite graphs with $\indm(G) = \ordm(G)=2$ when $\min\{|X|, |Y|\}\leq 4$.

\begin{corollary}\drlabel{corclosed}
For $2 \le m \le 4$,
the number $N(m,n)$ of connected non-isomorphic spanning subgraphs of $K_{m,n}$
with ind-match and ord-match both equal to 2 is:
$$
\begin{cases}
E, & \hbox{if $m = 2$ and necessarily $n \ge 3$}; \\
3, & \hbox{if $m = n = 3$};\\
{n-1 \choose 2} +  T, & \hbox{if $m = 3 < n$};\\
14, & \hbox{if $m = n = 4$}; \\
{n-1 \choose 2} + E + U + V, & \hbox{if $m = 4 < n$}, \\
\end{cases}
$$
where
\begin{align*}
E &= \left\lfloor n \over 2 \right\rfloor
\left(
\left\lceil n \over 2 \right\rceil-1
\right), \\
T &= -{25 \over 144} - {n \over 12}  + {7n^2 \over 24}  + {n^3 \over 36}  + {(-1)^n \over 16} 
+ {3 + \sqrt 3 i \over 54}\left({-1 - \sqrt 3 i \over 2}\right)^n 
+ {3 - \sqrt3 i \over 54} \left({-1 + \sqrt3 i \over 2}\right)^n, \\
%T &=& {-25 -12 n +42 n^2 + 4n^3 \over 144}  + {(-1)^n \over 16} 
%+ {3 + \sqrt 3 i \over 54}\left({-1 - \sqrt 3 i \over 2}\right)^n 
%+ {3 - \sqrt3 i \over 54} \left({-1 + \sqrt3 i \over 2}\right)^n, \\
U &=
{-1 \over 16} + {5 \over 12} n
+{5 \over 8} n^2+{1 \over 12} n^3+ {1 \over 16} (-1)^n,
\\
V &=
{-641 -486 n + 996 n^2 + 132 n^3 + 6 n^4 \over 3456} + {(11 + 2n)(-1)^n \over 128}
+ {(1-i) i^n + (1+i) (-i)^n \over 32} \\
&\hskip5em+ {1 + \sqrt 3 i \over 54} \left({-1-\sqrt 3 i \over 2} \right)^n
+ {1 - \sqrt 3 i \over 54} \left({-1+\sqrt 3 i \over 2} \right)^n.
%\\
\end{align*}
\end{corollary}

\proof
The justification for $m = 2$ is that there is only one $k$-sequence with $z = 2$,
and \Cref{propdisjoint} applies.

The cases $m = n = 3$ and $m = n = 4$ are special because
the two vertex sets can be switched,
and the general formula does not accommodate such switching.
These cases can be verified manually and we do not show the work here.

For $m = 3 < n$,
the $k$-sequences are $3,1,0$ and $3,2,1,0$.
For the first $k$-sequence,
the sets $J_1$ and $J_2$ are disjoint,
so \Cref{propdisjoint} applies
and we get ${n-1 \choose2}$ graphs.
For the $k$-sequence $3,2,1,0$, the number of graphs equals the sum
$$
\sum_{i = 1}^{n-1}
\sum_{j = 1}^{\min\{n-i, i\}}
\sum_{k = 0}^{\min\{n-i-j, j\}}
1,
$$
which by \Cref{proprecursive} equals~$T$.

For $m = 4 < n$,
the $k$-sequence $4,1,0$ contributes ${n-1 \choose 2}$ graphs
and the $k$-sequence $4,2,0$ contributes $E$ graphs
by \Cref{propdisjoint}.
The contribution of the $k$-sequence $4,2,1,0$ is 
$$
\sum_{i = 1}^{n-1}
\sum_{j = 1}^{n-i}
\sum_{k = 0}^{\min\{n-i-j, j\}}
1,
$$
which by \Cref{proprecursive4} equals~$U$.
Finally,
the $k$-sequence $4,3,2,1,0$ contributes
$$
\sum_{i = 1}^{n-1}
\sum_{j = 1}^{\min\{n-i, i\}}
\sum_{k = 0}^{\min\{n-i-j, j\}}
\sum_{l = 0}^{\min\{n-i-j-k, k\}}
1,
$$
which by \Cref{proprecursive4long} equals~$V$.
\qed

\section{Appendix: Some explicit summations}\label{sec:appendix}

The closed forms of the sums in this appendix are used in the previous section.

\begin{proposition}\drlabel{propsumdisjoint}
For any integer $n \ge 3$,
$$
\sum_{i=1}^{n-2} \sum_{j=1}^{\min\{i, n-i-1\}} 1
= 
\left\lfloor \frac{n}{2} \right\rfloor
\left(
\left\lceil \frac{n}{2} \right\rceil
-1
\right).
$$
\end{proposition}

\proof
Inequality
$i \le n-i-1$ holds if and only if $i \le \lfloor{n-1 \over 2}\rfloor$,
so the count simplifies to:
\begin{align*}
\sum_{i=1}^{n-2}
\sum_{j=1}^{\min\{i, n-i-1\}} 1 &=
\sum_{i=1}^{\lfloor{n-1 \over 2}\rfloor}
\sum_{j=1}^i 1
+ \sum_{i = \lfloor{n-1 \over 2}\rfloor+1}^{n-2}
\sum_{j=1}^{n-i-1} 1 \\
&=
\sum_{i=1}^{\lfloor{n-1 \over 2}\rfloor} i
+ (n-1)\left(n-2- \left\lfloor{n-1 \over 2}\right\rfloor\right)
- \sum_{i = \lfloor{n-1 \over 2}\rfloor+1}^{n-2} i \\
%&=
%\left\lfloor{n-1 \over 2}\right\rfloor
%\left(\left\lfloor{n-1 \over 2}\right\rfloor + 1 \right)
%- {1 \over 2} (n-2) (n-1)
%+ (n-1)\left(n-2- \left\lfloor{n-1 \over 2}\right\rfloor\right)
%\\
%%%
&=
\left\lfloor{n-1 \over 2}\right\rfloor
\left(\left\lfloor{n-1 \over 2}\right\rfloor + 1 \right)
+ (n-1)\left({n \over 2} - \left\lfloor{n-1 \over 2}\right\rfloor - 1\right) \\
&= \left\{ \begin{array}{ll}
    k^2 - k, & \text{ if } n = 2k; \\
    k^2, & \text{ if } n = 2k+1,
\end{array} \right. \\
&= \left\lfloor n \over 2 \right\rfloor^2
+
\left\lfloor n \over 2 \right\rfloor
\left(
\left\lceil n \over 2 \right\rceil
- \left\lfloor n \over 2 \right\rfloor
-1
\right) \\
&=
\left\lfloor n \over 2 \right\rfloor
\left(
\left\lceil n \over 2 \right\rceil
-1
\right). &\qedbox
\end{align*}

\begin{proposition}\drlabel{proprecursive}
For any non-negative integer $n$,
\begin{align*}
\sum_{i = 1}^{n-1}
&\sum_{j = 1}^{\min\{n-i, i\}}
\sum_{k = 0}^{\min\{n-i-j, j\}}
1 \\
&=
-{25 \over 144} - {n \over 12}  + {7n^2 \over 24}  + {n^3 \over 36}  + {(-1)^n
\over 16} \\
&\hskip 5em=
+ {3 + \sqrt 3 i \over 54}\left({-1 - \sqrt 3 i \over 2}\right)^n 
+ {3 - \sqrt3 i \over 54} \left({-1 + \sqrt3 i \over 2}\right)^n.
\end{align*}
\end{proposition}

\proof
The summation on the left is the number $a_n$ of quadruples of non-negative integers $i,j,k,l$ 
that add to $n$
with further restrictions that $i \ge j \ge 1$ and $j \ge k$.
Set $u = i - j$, $v = j - k$.
So we are looking for the number $b_n$ of non-negative integer solutions $u,v,k,l$
such that $u + 2v + 3k + l = n$
minus the number $c_n$ of integer solutions $i \ge 0, j = 0 = k, l$ with $i + l = n$.
%Let $b_n$ denote the number of non-negative integer solutions of $3x_1+2x_2+x_3+x_4=n$ and let
%$c_n$ denote the number of non-negative integer solutions of $x_1+x_2=n$. Set $a_n:=b_n-c_n$.

Let $A$ be the set of non-negative integer solutions of $3x_1+2x_2+x_3+x_4=n$. For each $i=1,2,3,4$, let $A_i$ be the set of non-negative integer solutions of $3x_1+2x_2+x_3+x_4=n$ with $x_i\geq 1$. Then $$A=A_1\cup A_2\cup A_3\cup A_4.$$ Using inclusion-exclusion formula $$|A|=\sum_{i=1}^4|A_i|-\sum_{i<j}|A_i\cap A_j|+\sum_{i<j<k}|A_i\cap A_j\cap A_k|-|A_1\cap A_2\cap A_3 \cap A_4|.$$
To compute the cardinality of $A_1$, set $y_1=x_1-1$. Since for any solution in $A_1$, we have $x_1\geq 1$, hence, for any such a solution, $y_1\geq 0$. Moreover, it follows from $3x_1+2x_2+x_3+x_4=n$ that $3y_1+2x_2+x_3+x_4=n-3$. Using induction hypothesis the number of solution of this equation is $b_{n-3}$. Therefore, $|A_1|=b_{n-3}.$ Similarly, one an compute the cardinality of other~$A_i$ and their intersections. Hence,$$b_n=2b_{n-1}-b_{n-3}-b_{n-4}+2b_{n-6}-b_{n-7}.$$
A similar argument shows that $$c_n=2c_{n-1}-c_{n-2}.$$
This yields that 
\begin{align*}
c_n &= 2c_{n-1}-c_{n-2} \\
&= 2c_{n-1}-2c_{n-3}+c_{n-4} \\
&= 2c_{n-1}-c_{n-3}-c_{n-3}+c_{n-4}\\ &=2c_{n-1}-c_{n-3}-2c_{n-4}+c_{n-5}+c_{n-4}\\
&= 2c_{n-1}-c_{n-3}-c_{n-4}+c_{n-5}\\ & =2c_{n-1}-c_{n-3}-c_{n-4}+2c_{n-6}-c_{n-7}.
\end{align*}
Therefore, $c_n$ satisfies the same recursive formula as $b_n$. Hence, their subtraction $a_n=b_n-c_n$ also satisfies the same formula $a_n=2a_{n-1}-a_{n-3}-a_{n-4}+2a_{n-6}-a_{n-7}$.

The roots of the corresponding characteristic polynomial $x^7-2x^6+x^4+x^3-2x+1$ are 
$1, 1, 1, 1, -1, -1/2 + \sqrt 3 i/2$ and $-1/2 - \sqrt 3 i/2$.
Thus by standard theory of linear recursive sequences with constant coefficients (see for example Theorem 6.21 and Remark 6.23 in \cite{LN}),
the closed form for $a_n$ equals
$$
a_n = c_0 + c_1 n + c_2 n^2 + c_3 n^3 + c_4 (-1)^n
+ c_5 \left((-1+\sqrt 3 \,i)/2\right)^n
+ c_6 \left((-1-\sqrt 3 \,i)/2\right)^n
$$
for some coefficients $c_0$ through $c_6$.
With that we can set up a system of linear equations
with manually computed $a_0 = 0$, $a_1 = 0$,
$a_2 = 1$, $a_3 = 3$, $a_4 = 6$, $a_5 = 10$, $a_6 = 16$,
and with the help of Macaulay2 \cite{GS} we computed the corresponding closed form for $a_n$ to be
$$
-{25 \over 144} - {n \over 12}  + {7n^2 \over 24}  + {n^3 \over 36}  + {(-1)^n \over 16} 
+ {3 + \sqrt 3 i \over 54}\left({-1 - \sqrt 3 i \over 2}\right)^n 
+ {3 - \sqrt3 i \over 54} \left({-1 + \sqrt3 i \over 2}\right)^n. \eqed
$$

\begin{proposition}\drlabel{proprecursive4}
For any non-negative integer $n$,
$$
\sum_{i = 1}^{n-1}
\sum_{j = 1}^{n-i}
\sum_{k = 0}^{\min\{n-i-j, j\}}
1 = 
{-1 \over 16} + {5 \over 12} n
+{5 \over 8} n^2+{1 \over 12} n^3+ {1 \over 16} (-1)^n.
$$
\end{proposition}

\proof
With methods as in the proof of \Cref{proprecursive},
we get
the recursive formula
$a_n = 3 a_{n-1} -2 a_{n-2} - 2 a_{n-3} + 3 a_{n-4} - a_{n-5}$.
The corresponding polynomial is
$x^5 - 3 x^4 + 2 x^3 + 2 x^2 - 3 x + 1$,
%x^5 - 3*x^4 + 2*x^3 + 2*x^2 - 3*x + 1
whose roots are $1,1,1,1,-1$.
The initial conditions can easily be computed to be
$a_1 = 0$, $a_2 = 1$, $a_3 = 4$, $a_4 = 9$, $a_5 = 17$,
%b = transpose matrix{{0,1,4,9,17}}
%A = matrix apply(5, i -> {1, i, i^2, i^3, (-1)^i}) | b
from which with linear algebra (or via {\tt https://oeis.org/A005744})
we get the expression
$$
a_n =
{-1 \over 16} + {5 \over 12} n
+{5 \over 8} n^2+{1 \over 12} n^3+ {1 \over 16} (-1)^n.
\eqed
$$

\begin{proposition}\drlabel{proprecursive4long}
For any non-negative integer $n$,
\begin{align*}
\sum_{i = 1}^{n-1}
&\sum_{j = 1}^{\min\{n-i, i\}}
\sum_{k = 0}^{\min\{n-i-j, j\}}
\sum_{l = 0}^{\min\{n-i-j-k, k\}}
= {-641 -486 n + 996 n^2 + 132 n^3 + 6 n^4 \over 3456} \\
&\hskip5em + {(11 + 2n)(-1)^n \over 128}
+ {(1-i) i^n + (1+i) (-i)^n \over 32} \\
&\hskip5em- {1 \over 27} \left({-1-\sqrt 3 i \over 2} \right)^{n+1}
- {1 \over 27} \left({-1+\sqrt 3 i \over 2} \right)^{n+1}.
\end{align*}
\end{proposition}

\proof
With the methods as in the proof of \Cref{proprecursive}
we get the recursive formula
$$
a_n = 2 a_{n-1} - a_{n-3} - 2 a_{n-5} + 2 a_{n-6} + a_{n-8} - 2 a_{n-10} + a_{n-11}.
$$
%2*a(n-1) - a(n-3) - 2*a(n-5) + 2*a(n-6) + a(n-8) - 2*a(n-10) + a(n-11)
%x^11 - 2*x^10 + x^8 + 2*x^6 - 2*x^5 - x^3 + 2*x -1
The roots of the corresponding polynomial
$x^{11} - 2 x^{10} + x^8 + 2 x^6 - 2 x^5 - x^3 + 2 x -1$
are $1,1,1,1,1,-1,-1,i,-i,(-1+\sqrt 3 \,i)/2,(-1-\sqrt 3 \,i)/2$.
The manually computed initial conditions are
$a_1 = 0$, $a_2 = 1$, $a_3 = 3$, $a_4 = 7$, $a_5 = 12$, $a_6 = 20$,
$a_7 = 30$, $a_8 = 44$, $a_9 = 61$, $a_{10} = 83$, $a_{11} = 109$.
Thus with linear algebra we get the closed form
\begin{align*}
a_n &= {-641 -486 n + 996 n^2 + 132 n^3 + 6 n^4 \over 3456} + {(11 + 2n)(-1)^n
\over 128} \\
&\hskip 5mm
+ {(1-i) i^n + (1+i) (-i)^n \over 32} \\
&\hskip5em+ {1 + \sqrt 3 i \over 54} \left({-1-\sqrt 3 i \over 2} \right)^n
+ {1 - \sqrt 3 i \over 54} \left({-1+\sqrt 3 i \over 2} \right)^n. & \qedbox \cr
\end{align*}

%\backmatter

%\vskip 2mm\noindent{\bf Disclosure statement.}
%No potential conflict of interest was reported by the authors.

\bibliographystyle{amsplain-ac}

\bibliography{JSFSY-AlCo}

\providecommand{\bysame}{\leavevmode\hbox to3em{\hrulefill}\thinspace}
\providecommand{\MR}{\relax\ifhmode\unskip\space\fi MR }
\providecommand{\MRhref}[2]{%
  \href{http://www.ams.org/mathscinet-getitem?mr=#1}{#2}
}
\providecommand{\href}[2]{#2}
\begin{thebibliography}{10}

\bibitem{BBH19}
Arindam Banerjee, Selvi~Kara Beyarslan, and H\`a~Huy T\`ai, \emph{Regularity of
  edge ideals and their powers}, Advances in Algebra \textbf{277 of Springer
  Proc. Math. Stat.} (2019), 17--52.

\bibitem{selviha}
Selvi Beyarslan, Huy~T\`ai H\`a, and Tr\^an~Nam Trung, \emph{Regularity of
  powers of forests and cycles}, J. Algebraic Combin. \textbf{{42} (4)} (2015),
  1077--1095.

\bibitem{Brod79}
Markus Brodmann, \emph{The asymptotic nature of the analytic spread}, Math.
  Proc. Cambridge Philos. Soc. \textbf{{86} (1)} (1979), 35--39.

\bibitem{CW05}
Kathie Cameron and Tracy Walker, \emph{The graphs with maximum induced matching
  and maximum matching the same size}, Discrete Math. \textbf{{299} (1-3)}
  (2005), 49--55.

\bibitem{CV11}
Alexandru Constantinescu and Matteo Varbaro, \emph{Koszulness, {K}rull
  dimension, and other properties of graph-related algebras}, J. Algebraic
  Combin. \textbf{{34} (3)} (2011), 375--400.

\bibitem{CHT}
S.~Dale Cutkosky, J\"urgen Herzog, and Ng\^o~Vi\^et Trung, \emph{Asymptotic
  behaviour of the {C}astelnuovo--{M}umford regularity}, Compositio Math.
  \textbf{{118} (3)} (1999), 243--261.

\bibitem{SF16}
Seyed Amin~Seyed Fakhari, \emph{Depth, {S}tanley depth, and regularity of
  ideals associated to graphs}, Arch. Math. (Basel) \textbf{107} (2016),
  461--471.

\bibitem{FAKHARI17}
\bysame, \emph{Depth and {S}tanley depth of symbolic powers of cover ideals of
  graphs}, J. Algebra \textbf{492} (2017), 402--413.

\bibitem{SY23}
Seyed Amin~Seyed Fakhari and Siamak Yassemi, \emph{Improved bounds for the
  regularity of powers of edge ideals of graphs}, J. Commut. Algebra
  \textbf{{15} (1)} (2023), 85--98.

\bibitem{FH}
Christopher~A. Francisco and Huy~T\`ai H\`a, \emph{Whiskers and sequentially
  {C}ohen-{M}acaulay graphs}, J. Combin. Theory Ser. A \textbf{{115} (2)}
  (2008), 304--316.

\bibitem{GRV05}
Isidoro Gitler, Enrique Reyes, and Rafael~H. Villarreal, \emph{Blowup algebras
  of ideals of vertex covers of bipartite graphs}, Contemp. Math. \textbf{376}
  (2005), 273--279.

\bibitem{GS}
Daniel Grayson and Michael Stillman, \emph{Macaulay2, a software system for
  research in algebraic geometry}, available at {\tt
  http://www.math.uiuc.edu/Macaulay2}.

\bibitem{hh2010}
J\"urgen Herzog and Takayuki Hibi, \emph{Monomial {I}deals}, Graduate Texts in
  Mathematics. Springer-Verlag, New York, 2010.

\bibitem{HHKO15}
Takayuki Hibi, Akihiro Higashitani, Kyouko Kimura, and Augustine O'Keefe,
  \emph{Algebraic study on {C}ameron-{W}alker graphs}, J. Algebra \textbf{422}
  (2015), 257--269.

\bibitem{HHKT16}
Takayuki Hibi, Akihiro Higashitani, Kyouko Kimura, and Akiyos Tsuchiya,
  \emph{Dominating induced matchings of finite graphs and regularity of edge
  ideals}, J. Algebraic Combin. \textbf{{43} (1)} (2016), 173--198.

\bibitem{HKMT21}
Takayuki Hibi, Kyouko Kimura, Kazunori Matsuda, and Akiyoshi Tsuchiya,
  \emph{Regularity and a-invariant of {C}ameron--{W}alker graphs}, J. Algebra
  \textbf{584} (2021), 215--242.

\bibitem{HKTT17}
Le~Tuan Hoa, Kyouko Kimura, Naoki Terai, and Tran~Nam Trung, \emph{Stability of
  depths of symbolic powers of {S}tanley--{R}eisner ideals}, J. Algebra
  \textbf{473} (2017), 307--323.

\bibitem{K06}
Mordechai Katzman, \emph{Characteristic-independence of {B}etti numbers of
  graph ideals}, J. Combin. Theory Ser. A \textbf{{113} (3)} (2006), 435--454.

\bibitem{vijay}
Vijay Kodiyalam, \emph{Asymptotic behaviour of {C}astelnuovo--{M}umford
  regularity}, Proc. Amer. Math. Soc. \textbf{128} (2000), 407--411.

\bibitem{LN}
Rudolf Lidl and Harald Niederreiter, \emph{Introduction to finite fields and
  their applications}, Cambridge University Press, 1986.

\bibitem{adam}
Adam~Van Tuyl, \emph{Sequentially {C}ohen-{M}acaulay bipartite graphs: vertex
  decomposability and regularity}, Arch. Math. (Basel) \textbf{{93} (5)}
  (2009), 451--459.

\bibitem{VV08}
Adam~Van Tuyl and Rafael Villarreal, \emph{Shellable graphs and sequentially
  {C}ohen-{M}acaulay bipartite graphs}, J. Combin. Theory, Ser. A \textbf{115}
  (2008), 799--814.

\bibitem{V90}
Rafael~H. Villarreal, \emph{Cohen-{M}acaulay graphs}, Manuscripta Math.
  \textbf{{66} (3)} (1990), 277--293.

\bibitem{W14}
Russ Woodroofe, \emph{Matchings, coverings, and {C}astelnuovo--{M}umford
  regularity}, J. Commut. Algebra \textbf{{6} (2)} (2014), 287--304.

\end{thebibliography}

\end{document}